%% file: ReviewMF.tex
\documentclass{article}
\usepackage{amsmath,amsfonts,amssymb,amsthm,amsopn,amscd}
\usepackage{bbm}
\usepackage{graphicx}
\usepackage{epsfig}
\usepackage{color}
\usepackage{colortbl}
\usepackage{authblk}
\usepackage{ucs} 
\usepackage[utf8x]{inputenc}

\definecolor{violet}{rgb}{0.50,0.00,.50}	

\input macros

\graphicspath{{Figures/}}
\begin{document}

\title{Stochastic firing rate models}

\author[1,2]{Jonathan Touboul}
\author[2]{Bard Ermentrout}
\author[3]{Olivier Faugeras}
\author[3,4]{Bruno Cessac}
\affil[1]{\small Department of Mathematical Physics, The Rockefeller University, New York, USA}
\affil[2]{\small Department of Mathematics, University of Pittsburgh, Pittsburgh, USA}
\affil[3]{\small NeuroMathComp Laboratory, INRIA, Sophia Antipolis, France}
\affil[4]{\small Laboratoire J. Dieudonn\'e, Universit\'e de Nice, France}

\maketitle

%
%
%
%

\section*{Abstract}
We review a recent approach to the mean-field limits in neural networks that takes into account the stochastic nature of input current and the uncertainty in synaptic coupling. This approach was proved to be a rigorous limit of the network equations in a general setting, and we express here the results in a more customary and simpler framework. We propose a heuristic argument to derive these equations providing a more intuitive understanding of their origin. These equations are characterized by a strong coupling between the different moments of the solutions. We analyse the equations, present an algorithm to simulate the solutions of these mean-field equations, and investigate numerically the equations. In particular, we  build a bridge between these equations and Sompolinsky and collaborators approach~\cite{sompolinsky-crisanti-etal:88,crisanti-sommers-etal:90}, and show how the coupling between the mean and the covariance function deviates from customary approaches.

\section*{Introduction}
The problem of modeling neural activity at scales integrating the effect of thousands of neurons has been a great endeavor in neurosciences for several reasons. First, it corresponds to the mesoscopic scale that most imaging techniques are able to measure. The activity monitored by these classical imaging techniques (such as the electro-encephalogram, the magneto-encephalogram, the optical imaging, etc\ldots) shows the global behavior of a whole area that results from the activity of several hundreds to several hundreds of thousands of neurons. Second, anatomical data recorded in the cortex reveal the existence of structures, such as the cortical columns, with a diameter of about $50 \mu m$ to $1 mm$, containing of the order of one hundred to one hundred thousand neurons belonging to a few different types. These columns exhibit specific functions as a result of the collective behavior of several neurons, as evidenced for instance in the primary visual cortex V1, where cortical columns are organized in a specific geometric configuration and respond to preferential orientations of bar-shaped visual stimuli. In this case, information processing does not occur at the  scale of individual neurons but rather emerges from the collective dynamics of many interacting neurons contributing to a mesoscopic or macroscopic signal. The description of this collective dynamics requires models which are different from individual neuron models. Indeed, when the number of neurons is large enough, averaging effects  appear, and the collective dynamics is well described by an effective mean-field, summarizing the effect of the interactions of a neuron with the other neurons, and  depending on a few effective parameters.

Models including these features which use population rates or activity equations are the foundation of a common approach to modeling for neural networks. These equations provide meanfield dynamics for the firing rate or activity of neurons within a network given some connectivity. This approach was introduced by Wilson and Cowan in the 1970s \cite{wilson-cowan:72,wilson-cowan:73}, and has been widely studied since then. This approach assumes that the activity coming from a cortical column is captured in the mean-firing rate of the cells in the population.  

One of the principal limitations of firing rate models is that they only take into account mean firing rates and do not include any information about higher order statistics such as correlations between firing activity of different neurons of the same populations or across populations. We know that spike trains of individual cortical neurons {\em in vivo}  are very noisy and have interspike interval (ISI) distributions close to Poisson (see e.g. \cite{softky-koch:93}). However, at the level of a neural population, the question of the statistical distribution of the activity is still open, and its description has been the subject of a many studies in the last fifteen years. Numerous studies on integrate-and-fire networks have shown that under certain conditions, even though individual neurons would exhibit Poisson-like statistics, the neurons fire asynchronously so that the total population activity evolves according to a mean-field rate equation, with a characteristic activation or gain function \cite{abbott-vanvreeswijk:93,brunel:00,brunel-hakim:99}. Formally speaking, the asynchronous state only exists in the thermodynamic limit where the number of neurons tends to infinity (finite-size deviations from this limit are addressed in \cite{brunel-hakim:99}). Moreover, even if the asynchronous state exists it might not be stable. This has motivated a number of authors to analyze statistical correlations and finite-size effects in spiking networks \cite{ginzburg-sompolinsky:94,meyer-vanvreeswijk:02,mattia-del-giudice:02,elboustani-destexhe:09}. Recently, a few papers analyzed the effects of correlations by utilizing the master equation introduced by El Boustani and Destexhe \cite{elboustani-destexhe:09} and providing different moment truncations using statistical physics tools, such as path integrals \cite{buice-cowan:09} and the Van Kampen expansion \cite{bressloff:09}. These approaches enlighten a complex interplay between the mean activity and the mean correlations, that yield non-trivial qualitative effects, that begin to be uncovered (see \cite{touboul-ermentrout:10}). A complementary approach has been developed by Cai and collaborators based on a Boltzmann--like kinetic theory of integrate--and--fire networks \cite{rangan-kovacic:08,cai-tao:04}, which is itself an extension of population-density methods \cite{omurtag-knight:00,nykamp-tranchina:00}. A recent alternative approach proposes a mathematically derived mean-field equation on stochastic firing--rate networks \cite{faugeras-touboul-etal:09} taking into account variability in the connectivity weights. They rigorously derive the mean-field equation of the network when the number of neurons tends to infinity and study the solutions for the dynamic and stationary mean-field equations. 

The present article builds upon this latter approach, and aims to relate this abstract approach to  Wilson and Cowan deterministic mean-field equation \cite{wilson-cowan:72,wilson-cowan:73} on one hand, and few other stochastic methods on the other hand. To this purpose, we implement a simple yet very familiar model that fits the framework of \cite{faugeras-touboul-etal:09}. We propose, in this simpler case, a heuristic argument that yields the same mean-field equation they obtain, analyse these equations, describe an algorithm to compute numerically the unique solution of these equations, and finally, investigate numerically the solutions of these equations.

\section{Material and Method : The stochastic firing-rate model}
In this article we study a model that we call the \emph{stochastic firing-rate model}, a stochastic version of the customary firing rate models. This section is aimed to define this model.

\subsection{The firing-rate model}
The building block of the network we study here consists of firing-rate Wilson and Cowan models \cite{wilson-cowan:72,wilson-cowan:73}, where we take into account the stochastic nature of synaptic inputs. In details, we consider a network composed of $N$ neurons indexed by $i \in \{1,\,\ldots,\,N\}$ belonging to $P$ populations indexed by $\alpha \in \{1,\,\ldots,\, P\}$, and let us denote $N_{\alpha}$ to be the number of neurons in population $\alpha$. Each neuron is described by its membrane potential $V_i$. Alternatively, this model can be seen as a model of cortical area, where the interacting units are cortical columns that follow Wilson and Cowan equations.

The Wilson and Cowan model supposes that  each incoming spike provokes a postsynaptic potential when received, and that these postsynaptic potentials are summed linearly to produce the membrane potential $V_i(t)$. A single action potential from neuron  $j$ received at time $t^*$ by the neuron $i$ provokes the addition of  potential $PSP_{ij}(t-s)$, where $s$ is the time of the spike hitting the synapse and $t$ the time after the spike. We neglect the delays due to the distance travelled down the axon by the spikes.

Under the assumption that the post-synaptic potentials sum linearly, the average membrane potential of neuron $i$ is given by the sum of all postsynaptic potentials coming from the other neurons $j$, through the formula:
\begin{equation}\label{eq:PSPsum}
V_i(t)=\sum_{j,k} PSP_{ij}(t-t^k_j) = \int_{t_0}^t PSP_{ij}(t-s)\sum_{t_k^j} \delta(t_j^k-s)\,ds
\end{equation}
where the sum is taken over the arrival times $t_j^k$ of the spikes produced by the neurons $j$. When the number of neurons tends to infinity, or when the firing rate of neurons increase, the sum can be approximated by the mean firing rate of each neuron, to first order. More precisely, this approximation consists in supposing that since the number of spikes arriving from neuron $j$ between $t$ and $t+dt$ is $\nu_j(t)dt$, we can approximate equation \eqref{eq:PSPsum} by:

\begin{equation*}
V_i(t)=\sum_j \int_{t_0}^t PSP_{ij}(t-s) \nu_j(s)\,ds.
\end{equation*}

We now assume that the neuron received input for arbitrary long times, i.e. $t_0=-\infty$. Another very important assumption of the rate model is that the relation between the mean firing rate and the membrane potential can be expressed in the form $\nu_i(t)=S_i(V_i(t))$, where $S_i$ is sigmoidal function (called the firing rate, or squashing function, see e.g.~\cite{gerstner-kistler:02,dayan-abbott:01}, and this function can be approximated in some cases using an averaging method (see section \ref{sec:squas}).

\begin{equation*}
V_i(t)=
\sum_j \int_{-\infty}^t PSP_{ij}(t-s) S_j(V_j(s))\,ds,
\end{equation*}

We finally note that the postsynaptic potentials $PSP_{ij}$ can depend on several variables in order to account, for instance, for adaptation or learning. However, in this paper, we follow the common assumption, initially made by Hopfield in \cite{hopfield:84}, that although the sign and amplitude may vary, the post-synaptic potential has the same shape no matter which presynaptic population caused it. This leads to the relation
\[
PSP_{ij}(t)=J_{ij} g_i(t).
\]
$g_{i}$ represents the unweighted shape (called a g-shape) of the postsynaptic potentials and $J_{ij}$ is the strength of the postsynaptic potentials elicited by neuron  $j$ on neuron $i$. In this paper, we further assume the g-shape to be an exponential function, and that both the g-shape and the sigmoidal functions only depends on the population ($\alpha$) the neuron $i$ belongs to: $g_i(t)=g_\alpha(t)=e^{-t/\tau_{\alpha}} \mathbbm{1}_{t>0}$. Under these assumptions, the membrane potential of neuron $i$ from population $\alpha$ can be written as the solution of the following differential equation:
\[
\der{V_{\alpha}}{t} = -\frac{1}{\tau_{\alpha}} V_{\alpha} + \sum_{\beta=1}^P\sum_{j=1}^{N_{\beta}} S_{\beta}(V_j(t)).\]

This equation models the interactions between the neurons in the network. External inputs also have a great influence in the system. These inputs are modelled as the sum of a deterministic input and a noisy input described in section \ref{sec:rand}. The deterministic input, generally assumed to originate from the thalamus, accounts for sensory inputs and cortico-cortical deterministic activity. These inputs are assumed to be identical for all neurons in the same population, and are denoted $I_{\alpha}(t)$ for the input to the population $\alpha$. Therefore, we end up with the deterministic model:
\[
\der{V_{\alpha}}{t} = -\frac{1}{\tau_{\alpha}}\,V_{\alpha} + I_\alpha (t)+ \sum_{\beta=1}^P\sum_{j=1}^{N_{\beta}} S_{\beta}(V_j(t)).
\]
which corresponds to a case of the Wilson and Cowan system. We now include the effect of noise in this model.

\subsection{Randomness in the network}\label{sec:rand}

Cortical areas receive a constant bombardment of action potentials originating from different sources, ranging from sensory stimuli to motor areas and cortico-cortical activity. These signals are characterized by a high level of variability, which is a result of the randomness\footnote{Some authors consider this variability as a result of an underlying finite dimensional chaos, see e.g.~\cite{faure-korn:01,korn-faure:03}} of processes arising in neuronal phenomena (see e.g.\cite{softky-koch:93,shadlen-newsome:94}). This noise can have different origins. It is the result of contributions attributed, for instance, to thermal noise arising from the discrete nature of electric charge carriers, the behavior of ion channels, synaptic transmission failures and global network effects. We also include in this noise term all the unrelated activity either coming from external stimuli or from extra network cortico-cortical activity. This bombardment of ``random'' spikes results in a noisy current $n_i$ whose probability distribution only depends on the population the neuron belongs to. Using the classical diffusion approximation  (see e.g. \cite{gerstner-kistler:02b,destexhe-mainen-etal:98,touboul-faugeras:07b}), we model that the noisy current $n_i$ as a white noise, the formal differential of a centered Wiener process with standard deviation $f_{\alpha}$ for neurons in population $\alpha$, and we assume that the noise received by each neuron is independent of those received by the other neurons.

Moreover, the synaptic weights $J_{ij}$ are known up to a certain precision, and cannot be defined individually with high precision. The mean connectivity, and the standard deviation of the error done in the measurement can nevertheless be experimentally evaluated. The random synaptic weights are therefore assumed to be drawn from a Gaussian distribution of mean $\Jbab/N_{\beta}$ and standard deviation $\Jab/ \sqrt{N_{\beta}}$, and do not evolve (they are said to be frozen during the evolution). The scaling chosen ensures that the local interaction process $\sum_{j=1}^{N_\beta} J_{ij}S_\beta(V_j(t))$, summarizing the effects of the neurons in population $\beta$ on neuron $i$, has a mean and variance which do not depend on $N_\beta$ and are only controlled by the phenomenological parameters $\Jbab,\Jab$. We finally make the simplifying technical assumption that the $J_{ij}$ are independent\footnote{
It is however known that synaptic weights are indeed correlated (e.g. 
via synaptic plasticity mechanisms);  these correlations are built by dynamics 
via a complex interwoven  evolution between neurons and synapses dynamics and postulating the form of synaptic weight correlations requires, on theoretical grounds, a detailed investigation of the whole history of neuron-synapse dynamics.}

Under these assumptions, the membrane potential of neuron $i$ in population $\alpha$ satisfies the equation
\begin{equation}\label{eq:NetEq}
	dV_i(t) = \left ( -\frac{1}{\tau_{\alpha}} \, V_i(t) +\sum_{\beta=1}^P \sum_{j=1}^{N_\beta}  J_{ij} S_{\beta}(V_j(t)) + I_{\alpha}(t)\right)\, dt + f_{\alpha} dW_i(t)
\end{equation}

\subsection{Averaging Models and Firing rate functions}\label{sec:squas}

The nonlinear form of the firing rate function has important implications; there are many choices that various authors have used. The simplest is the step function, where the neuron fires maximally or not at all depending on whether the potential is above or below threshold. If one uses a statistical-mechanical approach to derive ‘mean-field’ equations \cite{amari:72} then this sharp step function is smoothed out and can be represented by one of the well-known squashing functions: 
\begin{equation*}
  \begin{cases}
    S(V)=\frac{S_{\max}}{1+e^{-(V-V_T)/V_{s}}} \qquad \text{or}\\
    S(V)=\frac{S_{\max}}{2} (1+\erf (-(V-V_T)/V_{s}))
  \end{cases}
\end{equation*}
The first one is often referred to as the logistic function, and the second as the Gaussian function; $V_T$ is an activation threshold and $V_s$ governs the slope of the sigmoid. These choices appear somehow subjective and arbitrary.

When the number of neuron increases, this squashing function can be computed explicitly in some simple cases, using averaging properties (see e.g. \cite{rinzel-frankel:92} for class II neurons and \cite{ermentrout:94} for class I neurons). This method can guide the choice of the squashing function. It is described in more detail in section \ref{sec:squashMF}.

\section{Theory and Calculations}
 
The network we consider is governed by the $N$ equations such as \eqref{eq:NetEq}. In these equations, the neuron $i$ of population $\alpha$ integrates external deterministic and stochastic inputs, and receives from the network neurons a current, which we call the \emph{microscopic interaction process}. The mean-field problem consists in finding an equation that governs the activity of the neurons, at the level of the network population, when the number of neurons tends to infinity. In the stochastic setting of this article, this question is understood as a limit in law, under the joint law of the connectivities and the Brownian motions. In this section, we first review a way to obtain the squashing functions, before deriving the mean-field equations for the networks we are interested in.

\subsection{Averaging method and squashing functions}\label{sec:squashMF}
The squashing function corresponds to the averaged effect of many interconnections. When the number of neurons increases, this squashing function can be computed explicitly in some simple case, using averaging properties when assuming slow synapses (see e.g. \cite{rinzel-frankel:92} for class II neurons and \cite{ermentrout:94} for class I neurons). This method can provide a choice of squashing function. Let us consider for instance the following system of coupled neurons: 
\begin{equation}\label{eq:systemaveraging}
  \begin{cases}
    C \der{V_j}{t} + I^{\text{ion}}_j (V_j, w_j ) = \sum_{k} \lambda_{jk} s_k(t) \, (V^{\text{rev}}_k - V_j ) + I^{\text{appl}}_j\\
    \der{w_j}{t}  = q(V_j , w_j )\\
    \der{s_j}{t}  = \varepsilon (\hat{s}_j(V_j)-s_j)
  \end{cases}
\end{equation}
The $j$th cell of the network is represented by its potential, $V_j$, and all of the 
auxiliary channel variables that make up the dynamics for the membrane, 
$w_j$. The term $I^{\text{appl}}_j$ is any tonic applied current. Finally, the synapses are modeled by simple first order dynamics and act to hyperpolarize, shunt, or depolarize 
the postsynaptic cell. $\varepsilon$ is a small parameter corresponding to the time scale of the synapses compared to the time scale of evolution of the membrane potential, and is assumed to be small. Associated with each neuron is a synaptic channel 
whose dynamics is governed by the variable $s_j$ which depends in a (generally nonlinear) manner on the somatic potential. Thus, one can think of $s_j$ as being the fraction of open channels due to the presynaptic potential. 
The functions $\hat{s}_j$ have maxima of 1 and minima of 0. The effective maximal 
conductances of the synapses between the cells are in the nonnegative numbers $\lambda_{jk}$ and the reversal potentials of each of the synapses are $V^{\text{rev}}_k$. The goal is to derive equations that involve only the $s_j$ variables and thus reduce 
the complexity of the model while retaining the qualitative features of the original.

The main idea is to exploit the smallness of $\varepsilon$ and thus invoke the averaging theorem on the slow synaptic equations. Each of the slow synapses is held constant and the membrane dynamics equations are solved for the potentials, $V_j (t; s_1 , \ldots , s_n )$. The potentials, of course, naturally depend on the values of the synapses. We will assume that the potentials are either constant or periodic with period $T (s_1 , \ldots , s_n )$. Once the potentials are found, one then averages the slow synaptic equations over one period of the membrane dynamics obtaining:
\[\der{s_j}{t} =  \varepsilon(S_j (s_1, \cdots, s_n)-s_j )\]
where 
\begin{equation}\label{eq:Sj} 
	S_j (s_1 , \ldots , s_n ) = \frac{1}{T_j (s_1 ,\ldots , s_n)}\int_0^{T_j (s_1 ,\ldots , s_n)} \hat{s}_j (V_j (t; s_1 , \ldots , s_n ))dt.
\end{equation}
\begin{remark}
	Note that if we assume that when the cell is not firing, the potential is below the threshold for the synapse, we obtain a model very close to the activity-based (see \cite{ermentrout:98}).
\end{remark}
This method allows derivation of the squashing function. The main assumption is that the detailed phase and timing of individual spikes is not important; that is, one is free to average over many spiking events. Rinzel and Frankel apply the same averaging method to derive equations for a pair of mutually coupled neurons that was motivated by an experimental preparation. In their paper, they require only cross connections with no self-self interactions, and they are able to numerically determine the potential as a function of the strength of the synapses. They use a class of membrane models that are called “class II” (see \cite{rinzel-ermentrout:89}) where the transition from rest to repetitive firing occurs at a subcritical Hopf bifurcations that appears, for instance, in the Hodgkin-Huxley model. This latter assumption implies that the average potential exhibits hysteresis as a function of the synaptic drive, i.e. the functions $\hat{s}_j (s_k)$ are multivalued  for some interval of values of $s_k$. Because of this hysteresis, they are able 
to combine an excitatory cell and an inhibitory cell in a network and obtain oscillations, which is impossible for smooth  $\hat{s}_j (s_k )$ without self-excitatory interactions (see, e.g. \cite{rinzel-ermentrout:89}).

Class I membranes are simpler, as proved by Ermentrout in \cite{ermentrout:94}. It is known (see e.g. \cite{ermentrout-kopell:86}) that the frequency of oscillation in Class I membranes evolves as the square root of the bifurcation parameter (denoted by $p$):
\begin{equation}\label{eq:omega}\omega=C(p^*)\sqrt{(p-p^*)^+} + O(\vert p-p^* \vert).\end{equation}
If we make the approximation that $\hat{s}_j$  are simple on/off function, that is equal to $0$ if the membrane is at rest (and thus below threshold) and $1$ if the voltage is above threshold. Let $\xi(p)$ denote the amount of time during one cycle that the potential is above threshold. Then the equation \eqref{eq:Sj}, depending on the critical parameter $p$ on which the potential depends, 
is simplified to 
\[S(p) = \frac{1}{T(p)}\int_0^{T(p)} \hat{s}(V (t; p))dt = \frac{\omega(p)\xi(p)}{2\pi}\]
where we have used the fact that $T = 2\pi/\omega$. A fortuitous property of class I 
membranes is that the time for which the spike is above threshold is largely 
independent of the period of the spike so that $\xi(p)$ is essentially constant. 
(This is certainly true near threshold and we have found it to be empirically 
valid in a wide parameter range.) Thus, combining this with \eqref{eq:omega}, we obtain 
the very simple squashing function: 
\[S(p)=C(p^*)\sqrt{(p-p^*)^+}.\]
where $C(p)$ and $p$ depend on the other parameters in the model.

The value $p^*$ and the constant $C(p^*)$ need to be evaluated. Ermentrout in \cite{ermentrout:94} provides closed-form expressions for these quantities in the case of Morris-Lecar membrane model, and obtains a good agreement with the full system in numerical simulations. The resulting equations have the very nice compact form in a two cells network:
\begin{equation}\label{eq:2DimSquash}
\begin{cases}
  \tau_e \der{s_e}{t}+s_e & = C_e \sqrt{(\lambda_{ee}s_e - g_e^*(\lambda_{ie}s_i ,I_e))^+}\\
  \tau_i \der{s_i}{t}+s_i & = C_i \sqrt{(\lambda_{ei}s_e - g_e^*(\lambda_{ii}s_i ,I_i))^+}
\end{cases}
\end{equation}
where  $g_e^*(g, I )$ is the two-parameter surface of critical values of the excitatory conductance.
Note that these squashing functions are not similar to the usual nonlinear transforms between the voltage and the firing rate. They are less smooth than the usual ones, and in particular, are not differentiable at $p=p^*$: the slope at the bifurcation is infinite. The differential equations, such as \eqref{eq:2DimSquash}, obtained with this squashing function do not satisfy the Lipschitz standard condition for the existence and uniqueness of solutions. Moreover, in one dimension. differential equations with a square-root nonlinearity are classical counter-examples to the uniqueness condition: there can exist multiple solutions starting from the same initial condition. Therefore, this mean-field method yields equations that are complex to investigate, both in the class I case where the function is multivalued as in the class II case where it is non-smooth. Note eventually that this approach yields a mean-field description of model \eqref{eq:systemaveraging} in the sense that it results from an averaging of the behavior of an assembly of cells, when the number of neurons tends to infinity.

\subsection{Derivation of the Mean-Field equations}
Considering now a network of firing rate neurons of type \eqref{eq:NetEq}, a rigorous proof of the existence of the mean-field limit and the characterization of this limit is provided in \cite{faugeras-touboul-etal:09}, where the authors derived their calculations from the work of Ben-Arous and Guionnet \cite{ben-arous-guionnet:97}. This proof is extremely technical, but can be approached with non-rigorous heuristic arguments. We develop the heuristic proof of the derivation of this mean-field equation because we believe it provides a better understanding of what the technical mean-field limit derived in \cite{faugeras-touboul-etal:09} physically corresponds to. 

The heuristic proof we propose here is based on Amari's local chaos hypothesis introduced in 1972 and then widely used \cite{amari:72,amari-yoshida-etal:77,sompolinsky-crisanti-etal:88,crisanti-sommers-etal:90,cessac:95,samuelides-cessac:07}, and can be expressed as follows:

\noindent{\bf Amari's Local Chaos Hypothesis: }{\it
	For $N$ is sufficiently large, all the $V_i$ are pairwise stochastically independent, are independent of the connectivity parameters $J_{ij}$, and have a common distribution population per population. }

\begin{proposition}\label{prop:Convergence}
	Under Amari's local chaos hypothesis, the microscopic interaction term $U^N_{\alpha \beta}$ corresponding  to the current from population $\beta$ on a neuron $i$ in population $\alpha$ defined by:
\begin{equation}\label{eq:MicroInter}
	U^N_{i \beta}:=\sum_{j=1}^{N_{\beta}} J_{ij} S_{\beta}\big(V_j(t)\big)
\end{equation}
has a distribution that only depends on the populations $\alpha$ and $\beta$, and converges in law (under the joint probability of the connectivity weights and the voltages) when the number of neurons tends to infinity towards an \emph{effective interaction process} $U^{V}_{\alpha \beta}$ such that:
 \begin{equation}\label{eq:effectiveInteractionProcessParams}
  \begin{cases}
   \Exp{U^{V}_{\alpha\beta}(t)} = \Jbab \mathbb{E}[S_{\beta}(V_{\beta}(t))] \\
   \Cov(U^{V}_{\alpha\beta}(t), U^{V}_{\gamma \delta}(s))  = \Jdab \mathbb{E}\Big[ S_{\beta}(V_{\beta}(t)) S_{\beta}(V_{\beta}(s))\Big] \mathbbm{1}_{\alpha=\gamma;\, \beta=\delta}
  \end{cases},
 \end{equation}
where $V_{\alpha}$ corresponds to the common law of the neurons in population $\alpha$ ensured by the local chaos hypothesis, and $\mathbbm{1}_A$ is the indicator function of the set $A$.
\end{proposition}

For the sake of simplicity, for any P-dimensional stochastic process $X$, we introduce the notation $m_{\beta}^X(t) = \mathbb{E}[S_{\beta}(X_{\beta}(t))]$ and
\[
\Debx(t,s) \eqdef \mathbb{E}\Big[S_{\beta}(X_{\beta}(t))S_{\beta}(X_{\beta}(s))\Big]
\]

We now turn to the heuristic proof of this proposition. We consider that $N$ is large enough for the local chaos hypothesis to be valid. Under this assumption, the voltages are independent, identically distributed and independent of the individual connectivity weights $J_{ij}$. 

The microscopic interaction process \eqref{eq:MicroInter} is therefore the sum of independent identically distributed random variables. The functional central limit theorem applies to our regular case, following \cite{pollard:90} (very weak regularity conditions are necessary on the summed processes to get this convergence), provided the convergence of the two first moments of the sum. 

The problem therefore reduces to the computation of the limits of the mean and standard deviation of the microscopic interaction process when the number of neurons tends to infinity. First of all, the mean of the microscopic interaction process for 
$N$ large enough is constant, since we have
\begin{align*}
 \mathbbm{E}\Big(U^N_{\alpha \beta}(t)\Big) &= \mathbbm{E}\Big( \sum_{j=1}^{N_{\beta}} J_{ij}S_{\beta}(V_j(t))\Big) = \sum_{j=1}^{N_{\beta}} \mathbbm{E}\Big( J_{ij}\Big) \mathbbm{E}\Big(S_{\beta}(V_j(t))\Big) \\
 &=  \Jbab \mathbbm{E}\Big(S_{\beta}(V_{\beta}(t))\Big) 
\end{align*}
This constant value is thus the limit of the means of the sum of independent random variables, and therefore the mean of the limiting Gaussian process. 

Similar straightforward computations (see appendix \ref{append:Moments}) prove that the covariance of the microscopic interaction process converges to the expression given in the proposition. From the functional central limit theorem and the convergence of the two first moments of the microscopic interaction, we conclude to the convergence of this sequence of processes to the effective Gaussian process of mean and covariance given in the proposition. The statistical independence between this effective interaction process and the membrane potential $V$ is also proved in the same appendix \ref{append:Moments}.

As a direct consequence of this last proposition, the mean-field equation for the membrane potential processes can be derived.

\begin{proposition}
	Under the local chaos hypothesis, the network equations \eqref{eq:NetEq} converge in law towards the solution of the non-Markovian stochastic equation:
	\begin{equation}\label{eq:V_MFE}
		dV_{\alpha}(t) = \left( -\frac{1}{\tau_{\alpha}} \, V_{\alpha}(t) + \sum_{\beta=1}^P U^{V}_{\alpha\beta}(t) + I_{\alpha}(t)\right) \, dt + f_{\alpha}dW^{\alpha}_t.
	\end{equation}
	where the processes $(W_{\alpha}(t))_{t\geq t_0}$ are independent Wiener processes and $U^{V}(t) = (U_{\alpha\beta}^{V}(t);\; \alpha, \beta \in \{1,\, \ldots,\, P\})_{t}$ is the effective interaction process. 
\end{proposition}

\begin{proof}
	The solution of the network equation \eqref{eq:NetEq} with initial condition $V(0)$ at $t=0$ can be written as:
	\[V_{i}(t) = V_{\alpha}(0)e^{-t/\tau_{\alpha}} + \sum_{\beta=1}^{P} \int_0^t e^{(s-t)/\tau_{\alpha}} U^N_{\alpha\beta}(s) \,ds + f_{\alpha}\int_{0}^t e^{(s-t)/\tau_{\alpha}} \, dW^{\alpha}_s.\]
	Because of the convergence in law of the microscopic interaction process to the effective interaction process, we have the convergence in law of the integral term $\int_0^t e^{(s-t)/\tau_{\alpha}} U^N_{\alpha\beta}(s) \,ds$ towards the effective term $\int_0^t e^{(s-t)/\tau_{\alpha}} U^V_{\alpha\beta}(s) \,ds$, and therefore, for any neuron $i$ of population $\alpha$, the potential converges in law towards the solution $V_{\alpha}$ of the stochastic fixed-point equation:
	\begin{equation}\label{eq:V_MFE_Integ}
		V_{\alpha}(t) = V_{\alpha}(0)e^{-t/\tau_{\alpha}} + \sum_{\beta=1}^{P} \int_0^t e^{(s-t)/\tau_{\alpha}} U^V_{\alpha\beta}(s) \,ds + f_{\alpha}\int_{0}^t e^{(s-t)/\tau_{\alpha}} \, dW^{\alpha}_s
	\end{equation}
	which is equivalent to equation \eqref{eq:V_MFE}.
\end{proof}

The two equivalent equations \eqref{eq:V_MFE} and \eqref{eq:V_MFE_Integ} are referred to as the mean-field equations. In the limit where the number of neurons tends to infinity, all the neurons have the same distributions, behave independently and for any neuron in population $\alpha$, its membrane potential is solution of this set of equations.

\begin{remark}
	All the derivations have been done  based on Amari's local chaos hypothesis, which is a non-rigorous assumption. However, the exact same result is obtained by rigorous mathematical methods~\cite{faugeras-touboul-etal:09,ben-arous-guionnet:95,guionnet:97}.
\end{remark}

We discuss the form of these mean-field equations and the existence and uniqueness of their solutions in the following section. 

\subsection{Analysis of the Mean-Field equations}\label{sec:anal}

The mean-field equations derived in the previous section are rather unusual equations. Indeed, they involve the effective interaction process term, which is a functional of the solution of the equation. More precisely, it is a Gaussian process whose law depends on the statistics of the solution of the equation in a non Markovian way. Therefore, this equation cannot be considered as a stochastic differential equation. It is an equation set in the space of stochastic processes, or equivalently an equation on the probability distribution of the mean-field solution. Similar issues were observed in the case of spin glasses \cite{ben-arous-guionnet:95}. The existence and uniqueness of solutions of the dynamic mean-field equation (i.e. starting from a Gaussian initial condition) is proved in \cite{faugeras-touboul-etal:09}, by considering the mean-field equation as a fixed point equation in the set of stochastic processes. Conditions are also given for a stationary solution to exist. 

The approach is based on the fact that solutions having a Gaussian initial condition $V(0)$ (or for stationary solutions) are necessarily Gaussian processes, and therefore characterizing the mean $\mu^V(t)$ and the correlation function $C^V_{\alpha\beta}(s,t)$ is sufficient to identify the solution of the mean-field equations. How do we compute these two deterministic functions? Let us start by writing the equations they satisfy. By straightforward computations on equations \eqref{eq:V_MFE} and  \eqref{eq:V_MFE_Integ}, we obtain the equations for the mean $\mu^V(t)$ and the correlation $C_{\alpha\beta}(s,t)$ for $s\leq t$:
\begin{equation}\label{eq:MFE_Moments}
	\begin{cases}
	\der{\mu^V_{\alpha}}{t} &= -\frac{1}{\tau_{\alpha}}\mu^V_{\alpha}(t)
	+ \sum_{\beta=1}^P \Jbab  \int_{-\infty}^{+\infty} S_{\beta} \left ( x\sqrt{C_{\beta\beta}^V(t,t)} + \mub^V(t) \right) Dx
	+ \Ia(t),\\
	C_{\alpha \alpha}^V(t,s) & =e^{-(t+s)/\ta}\Big[\textrm{Var}(V_{\alpha}(0))+ \frac{\ta f_{\alpha}^2}{2}\left(e^{\frac{2s}{\ta}}-1\right) \\ 
	& \quad+ \sum_{\beta=1}^P \Jdab \int_{0}^t\int_{0}^s e^{(u+v)/\ta}\Delta_{\alpha\beta}^V(u,v)dudv\Big],\\
	C_{\alpha \beta}(s,t) & \equiv 0 \qquad \qquad \text{for }\; \alpha \neq \beta
\end{cases}
\end{equation}
where $Dx = \exp(-x^2/2)/\sqrt{2\pi}\,dx$ is the standard Gaussian measure and 
\begin{multline}\label{eq:delta}
	\Deabx(u,v) = \int_{\bbbr^2} S_{\beta}\left( x\frac{\sqrt{C_{\beta\beta}^X(u,u)\,C_{\beta\beta}^X(v,v) - C_{\beta\beta}^X(u,v)^2}}{\sqrt{C_{\beta\beta}^X(v,v)}}
	+y\frac{C_{\beta\beta}^X(u,v)}{\sqrt{C_{\beta\beta}^X(v,v)}} +\mu^X_\beta (u)\right) \\
 \quad \times S_{ \beta}\left(y \sqrt{C_{\beta\beta}^X(v,v)}+\mu^X_\beta (v)\right)\,Dx\,Dy\\
\end{multline}
 
In the more general setting of \cite{faugeras-touboul-etal:09}, it is proved that these equations have a unique solution that can be computed via the iteration of a map $\mathcal{F}$ defined on the set of continuous mean and covariance functions. This function transforms a Gaussian process $X$ with mean $\mu^X$ and covariance $C^X$ into the Gaussian process $Y=\mathcal{F}(X)$ with mean $\mu^Y$ and covariance $C^Y$ given by:

\begin{align}
	 \nonumber\muay(t) &=\muax(0)e^{-t/\ta}+\int_{0}^t e^{-(t-s)/\ta}(\sum_{\beta=1}^P \Jbab \Exp{S_{ \beta}(X_{\beta}(s))}+\Ia(s))ds \\
\nonumber & = \muax(0)e^{-t/\ta}+\int_{0}^t e^{-(t-s)/\ta} \Ia(s)ds  \\
\label{eq:moyenne} & \qquad + \sum_{\beta=1}^P \Jbab \int_{0}^t  e^{-(t-s)/\ta}
\int_{-\infty}^{+\infty} S_{\beta} \left ( x\sqrt{\vbx(s)} + \mub^X(s) \right) Dx ds.
\end{align}
where we denoted $\vax(s)$ the standard deviation of $X_\alpha$ at time $s$, instead of $\Cax(s,s)$, and 
\begin{multline}\label{eq:Variance}
\Cay(t,s)=e^{-(t+s)/\ta}\Big[v_\alpha^X(0)+ \frac{\ta\sa^2}{2}\left(e^{\frac{2s}{\ta}}-1\right) \\ 
+ \sum_{\beta=1}^P \Jdab \int_{0}^t\int_{0}^s e^{(u+v)/\ta}\Deabx(u,v)dudv\Big],
\end{multline}
where $\Deabx$ is defined in equation \eqref{eq:delta}.

These equations take a particularly simple form in the case where the firing-rate sigmoidal functions are $\erf$ functions, where $\erf(y)=\int_{-\infty}^y Dx$ is the repartition function of the standard Gaussian distribution. Let us consider that $S_{\alpha}(x) = \erf(g_{\alpha}x+\gamma_{\alpha})$ where $\erf$ is the repartition function of the Gaussian. In that case (see appendix \ref{append:Erf} for the calculations), the mean-field equations take the form:
\[\der{\mu_{\alpha}}{t} = -\frac{\mu_{\alpha}}{\tau_{\alpha}} + \sum_{\beta=1}^{P}\Jbab \erf\left(\frac{g_{p(i)}\, \mu_\beta(t) + \gamma_{p(i)}}{\sqrt{1+g^{2}v_\beta(t)}}\right)\]
and the equation governing the standard deviation of the process reads:
\begin{multline}
v_\alpha(t) = e^{-2t/\tau_{\alpha}}  \Big [ v_\alpha(t_0) + \frac{\tau_{\alpha}f_{\alpha}^{2}} {2} (e^{2t/\tau_{\alpha}} - 1 ) +
\sum_{{\beta=1}}^{P} \sigma_{\alpha \beta}^{2} \int_{t_{0}}^{t}\int_{t_{0}}^{t} e^{(u+v)/\tau_{\alpha}}\\
\int_{\R} S(\sqrt{v_\beta(t)} y + \mu_{\beta}(t)) \,S\left ( \frac{C_{\beta\beta}(t,s) y + \mu_{\beta}(s)\sqrt{v_\beta(t)}} {v_\beta(t) + g^{2} (v_\beta(t)v_\beta(s)-C_{\beta \beta}(t,s)^2)}\right)Dy.
\Big]
\end{multline}

These simpler equations allow for  more efficient numerical simulations, and an easier analytical treatment, and will be discussed in the Results section. 

\subsection{Numerical simulations algoritm}

The mean-field equations can be seen as a fixed point equation in the space of stochastic processes, or equivalently as a fixed point equation on the mean and the covariance function. In \cite{faugeras-touboul-etal:09}, the authors show that this solution can be obtained through the iteration of a function that operates on the mean and on the covariance function. This is the approach we chose to numerically compute the solutions of the mean-field equations. 

\subsubsection{Description of the algorithm}
In details, let $X$ be a $P$-dimensional Gaussian process of mean 
$\mu^X = (\muax(t))_{\alpha=1\ldots P}$ and 
covariance $C^X = (C^X_{\alpha\beta}(s,t))_{\alpha, \beta \in \{1\ldots P\}}$. We use the map $\mathcal{F}$ which makes the mean-field equation a fixed-point equation, as described in section \ref{sec:anal}.

From equation \eqref{eq:moyenne}, we can see that knowing $\vax(s), s \in [0,t]$ we can compute $\muay(t)$ using a standard discretization scheme of the integral, with a small time step compared with $\ta$ and the characteristic time of variation of the input current $\Ia$. Alternatively, since $\muay$ satisfies the differential equation:
\begin{equation*}
 \frac{d\muay}{dt} =-\frac{\muay}{\ta}  
+ \sum_{\beta=1}^P \Jbab  \int_{-\infty}^{+\infty} S_{\beta} \left ( x\sqrt{\vbx(t)} + \mub^X(t) \right) Dx
+ \Ia(t),
\end{equation*}
we can compute faster and with a better accuracy the solution using a Runge-Kutta algorithm.

The covariance variable of the image of $\mathcal{F}$ (from equation \eqref{eq:Variance}) can be split into the sum of two terms: the external noise contribution $\Ca^{OU}(t,s)=e^{-(t+s)/\ta}\left[\vax(0)+ \frac{\ta\sa^2}{2}\left(e^{\frac{2s}{\ta}}-1\right)\right]$, where OU stands for Ornstein-Uhlenbeck, and the interaction between the neurons. The external noise contribution is a simple function and can be computed straightforwardly. To compute the interactions contribution to the standard deviation
we have to compute  the symmetric function of two-variables:
\[ \Habx(t,s)=e^{-(t+s)/\ta} \int_0^t\int_0^s e^{(u+v)/\ta}\Deabx(u,v)dudv,\]
from which one obtains the standard deviation using the formula
\begin{equation*}
\Cay(t,s)=\Ca^{OU}(t,s)+\sum_{\beta=1}^P \Jdab \Habx(t,s).
\end{equation*}
To compute the function $\Habx(t,s)$, we start from $t=0$ and $s=0$, where $\Habx(0,0)=0$. We only compute $\Habx(t,s)$ for $t>s$ because of the symmetry. It is straightforward to see that:
\begin{equation*}
\Habx(t+dt,s)=\Habx(t,s)\left[1-\frac{dt}{\ta} \right]+\Dabx(t,s)dt + o(dt),
\end{equation*}
with
\begin{equation*}
\Dabx(t,s)= e^{-s/\ta}\int_0^s e^{v/\ta}\Deabx(t,v)dv.
\end{equation*}
Hence computing $\Habx(t+dt,s)$ knowing $\Habx(t,s)$ amounts to computing $\Dab(t,s)$. Fix $t\geq 0$. We have $\Dab(t,0)=0$ and 
\begin{equation*}
\Dabx(t,s+ds)=\Dabx(t,s)(1-\frac{ds}{\ta})+\Deabx(t,s)ds + o(ds).
\end{equation*}
This algorithm enables us to compute $\Habx(t,s)$ for $t>s$. We deduce $\Habx(t,s)$ for $t<s$ using the symmetry of this function. Finally, to get the values of $\Habx(t,s)$ for $t=s$, we use the symmetry property of this function and get:
\begin{equation*}
\Habx(t+dt,t+dt)=\Habx(t,t)\left[1-\frac{2dt}{\ta} \right]+2\Dabx(t,t)dt + o(dt).
\end{equation*}

\subsubsection{Analysis of the algorithm}
\paragraph{Convergence rate}
Using the fact that the iterations of the algorithms produce a Cauchy sequence and using evaluations done in the proofs of existence and uniqueness of solutions provided in \cite{faugeras-touboul-etal:09}, one can show that the precision of the algorithm is given by 
\begin{equation}\label{eq:CVAlgo}
  \|C^{{\mathcal{F}^N}(X)} - C^{X_{MF}} \|_{\infty} \leq C_2\,(N+T-t_0)\, dt+ R_N(\tilde{k})
\end{equation}
 where $\mathcal{F}^N(X)$ is the $N$th iterate of the map $\mathcal{F}$ on the initial process $X$ and $X_{MF}$ the unique fixed point of $\mathcal{F}$, i.e. the mean-field solution. In equation \eqref{eq:CVAlgo}, $C_2$ denotes a constant depending on the parameters of the model, $R_N(x)$ is the exponential remainder, i.e. $R_N(x) = \sum_{n=N}^{\infty} x^n/n!$ and $\tilde{k}$ is some constant depending on the parameters of the model. 

\paragraph{Complexity}
The complexity of the algorithm depends on the complexity of the computations of the integrals. The algorithm described hence has the complexity $O(N_{\text{iter}} (\frac{T}{dt})^2)$ where $N_{\text{iter}}$ is the iteration number of the algorithm.

\section{Results}
In this section, we compare the results provided by the present approach to a variety of deterministic and stochastic approches. We will be particularly interested in comparing the solutions of the present mean-field equation to the deterministic Wilson and Cowan equations~\cite{wilson-cowan:72,wilson-cowan:73}:
 \[\der{V_{\alpha}}{t} = -\frac{1}{\tau_{\alpha}} \, V_{\alpha}(t) + \sum_{\beta=1}^P \Jbab S_{\beta}(V_{\beta}(t))\]

If $C_{\beta \beta}(t,t) \equiv 0$ for all $t \in \R$, then the equation the mean-field equations \eqref{eq:MFE_Moments} reduces to:
\begin{equation*}
 \frac{d \mu_\alpha(t)}{dt}= 
-\frac{\mu_\alpha(t)}{\tau_{\alpha}}+\sum_{\beta=1}^P \Jbab
S_\beta\left(\mu_\beta(t)\right) +I_\alpha(t),
\end{equation*}
which is precisely the customary Wilson and Cowan equations corresponding to the system where no randomness is considered neither in the input nor in the coefficients. This is precisely the limit for $f_{\alpha}=0$ and $\Jdab=0$ for all $\alpha$ and $\beta$ of the mean-field equations. Therefore, the Wilson and Cowan equations can be seen as the deterministic limit of the general mean-field equations presented in this paper. Moreover the solutions of the general mean-field equations in the small noise case converge to the solutions of these deterministic equations. This convergence is to be taken in a weak sense, since the conditions for existence and uniqueness of the solution of the mean-field equations are not satisfied in the case where no noise is present (the setting for existence and uniqueness necessitates to stay in the space of continuous mean and covariance functions).

However, when noise is considered, the mean of the solutions of the mean-field equations differs from the solutions of the Wilson and Cowan's system, and non-trivial differences appear. The precise study of the differences between the mean-field system and customary approaches is still an active ongoing research field, but evidences of non-trivial effects are already found. First of all, in the one-dimensional case, we obtained numerically that presence of noise can delay the apparition of bifurcations, stabilizing fixed points that were unstable in the corresponding Wilson and Cowan equation (see figure \ref{fig:Pitch}). This numerical observation is under
\begin{figure}
	\centering
		\includegraphics[width=.5\textwidth]{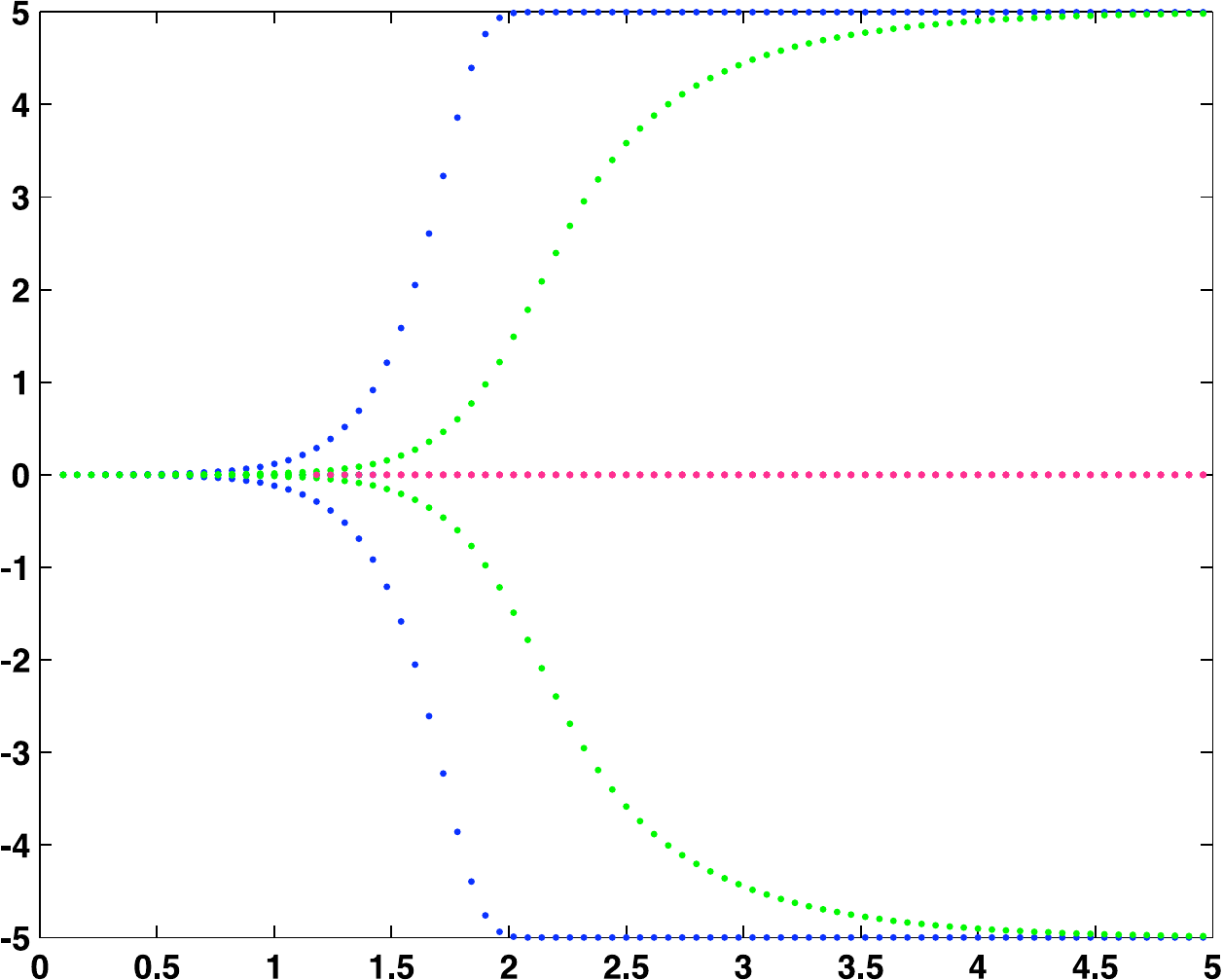}
	\caption{Delayed pitchfork bifurcation numerically simulated in a one population network. We plotted the value of the mean of the process on an interval $[T_1,T_2]$ with $T_1$ sufficiently large, as a function of the slope $g$ of the sigmoid $S(x)=1/(1+exp(-g\,x))$}
	\label{fig:Pitch}
\end{figure}
analytical investigation. Other non-trivial effects include transitions between a deterministic and a chaotic behavior, similar to the ones predicted analytically by Crisanti, Sompolinsky and colleagues  \cite{sompolinsky-crisanti-etal:88,crisanti-sommers-etal:90}, and that we study in section \ref{sec:Sompo}. Another non-trivial effect is the influence of the correlations in the oscillatory regime for a two populations case, presented in \ref{sec:Oscill}. 

\subsection{Stationary solutions in the one population case}\label{sec:Sompo}
The mean-field equations \eqref{eq:V_MFE} provide another approach to the analytical results of the extensive work of Crisanti, Sompolinsky and colleagues \cite{sompolinsky-crisanti-etal:88,crisanti-sommers-etal:90}. In these papers, the authors use a generating functional approach to derive the mean-field equations the system satisfies. The equations they obtain are very similar to ours, in the one population case. Hence so far we have exactly the same equations and formalism. They further simplify the problem by considering stationary solutions. In this particular case, the necessary conditions to get a stationary solutions to the mean-field equations stated in \cite{faugeras-touboul-etal:09} are satisfied, namely, since the time constants of evolution for each population and the standard deviation of the noise are constant, and what they call the leak matrix $L$, here the diagonal matrix with diagonal element $-1/\tau_{\alpha}$, has full rank and strictly negative eigenvalues. The stationary solutions can therefore be written as:
\begin{equation}\label{eq:LTMFE}
 V_{\alpha}(t) = I_{\alpha} + \sum_{\beta=1}^P \int_{-\infty}^t e^{-(t-s)/\tau_{\alpha}} U_{\alpha\beta}^{V}(t)\, dt + \sigma_{\alpha} \int_{-\infty}^t e^{-(t-s)/\tau_{\alpha}} \,dW_t^{(\alpha)}
\end{equation}
In the stationary case, the mean is constant, and the covariance $C_{\alpha\alpha}(t,s)$ only depends on the time difference $\tau = \vert t-s \vert$. They then write the fixed point equation on the covariance as a second order differential equation. This equation can be written as the mouvement of a particle in a potential well, this potential depends on a non-free parameter, $q$, which is equal to the standard deviation of the process (for a time difference $\tau=0$), namely $C(0)$. Because of the structure of stationary solutions, this parameter depends on the whole past of the solution on $(-\infty, 0)$, and the equation they write, of type:
\begin{equation}\label{EDOSomp}
\frac{d^2 C}{d\tau^2}=-\frac{\partial V_{q}}{\partial C}.
\end{equation} 
where $V_q$ is the potential depending on $q=C(0)$, remains a functional equation on the covariance, that do not differ from ours. Note that this dependency illustrates the non-Markovian nature of the problem. 

At this point, the authors treat the problem as if $q$ were a free parameter, which allows them to study the phase portrait of these equations and distinguish two regimes depending on the maximal slope $g$ of the sigmoidal firing function: for small values of $g$, a deterministic regime where the covariance has a unique fixed point equal to $0$, and for larger values of $g$, a homoclinic
trajectory connecting the point $q=C^\ast>0$ where
$V_q$ vanishes to the point $C=0$. The authors show that this is the only stable solution in the system, and interpret this solution as a chaotic solution in the neural network. Moreover, this solution can be found using energy conservation, and Taylor expansion of $V_q$ can be provided. These two predicted solutions are found by the simulation of the full mean-field system, as illustrated in Figure~\ref{fig:SompoCase}. In this figure, we represent the correlation function $C(s,t)$ as a function of $\tau=t-s$. Hence for each value of $\tau$ correspond several correlation values
\begin{figure}
	\centering
		\includegraphics[width=.9\textwidth]{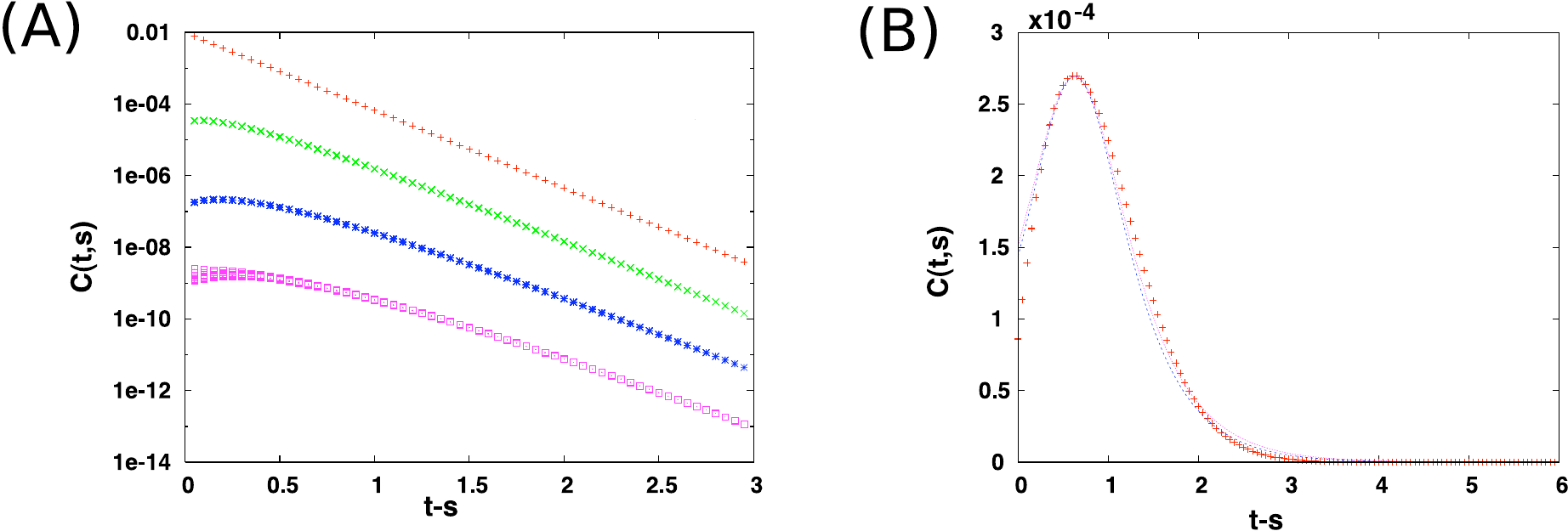}
	\caption{Simulations of the correlation equation in the stationary case. $\tau_1=0.25$, $\sigma_{1,1}=1$,  $S(x)=\tanh(g\,x)$, and $\bar{J}=1$. (A) Deterministic case: $g=0.5$. The graph shows the correlation for 1 (red), 2 (green), 3 (blue) and 4 (purple) iterations of the map $\mathcal{F}$ in a semi-log scale, and illustrate the convergence of the iterations towards zero as predicted by Sompolinsky et al.. (B) Chaotic case: $g=5$. We observe after convergence of the algorithm that the covariance function is non-vanishing (red crosses), and is well approximated by a fourth-order Taylor expansion of the theoretical solution (blue dotted curve) and even better by the sixth order expansion (purple curve). }
	\label{fig:SompoCase}
\end{figure}
which are superposed if the solution is stationary (resulting in a curve). These numerical results are compared to Sompolinsky's predictions and to the Taylor expansion to the fourth and sixth order. The authors of \cite{sompolinsky-crisanti-etal:88,crisanti-sommers-etal:90} predict that there exist a critical value of the slope of the sigmoid function $S$, denoted $g_c$, at which the system presents a sharp transition. This value in the numerical case we treat in figure~\ref{fig:SompoCase} is $g_c=4$. For any $g$ smaller than the critical value $g_c$, the system will present a deterministic behavior, with a covariance function identically equal to $0$. They call this solution the trivial case, and in that case the solutions of the mean-field equations reduce to the solutions of Wilson and Cowan system. Though they do not prove that the system indeed converges to this solution, they show that it exists when the parameter $q$ is free, this solution exists. We numerically compute, using the algorithm presented and integrating from a very large negative time, that indeed this solution is the one that is selected by the system, i.e. the only solution to the stationary mean-field equation. Similarly, for $g>g_c$, the covariance of the system converges towards the convergence corresponding to the homoclinic orbit of equation \eqref{EDOSomp} corresponding to the case where the parameter $q=C(0)$ is free. Therefore, the behaviors predicted by the studies of Crisanti, Sompolinsky and colleagues are confirmed by our numerical approach, and their approximation is therefore \emph{a posteriori} validated.

\subsection{Oscillations in a two-populations network}\label{sec:Oscill}

Let us now present a case  where the fluctuations of the Gaussian field 
act on the dynamics of $\mua(t)$ in a non trivial way, with a behavior departing
from the naive mean-field picture.
We consider  two interacting populations where the connectivity weights are Gaussian random variables $J_{\alpha\beta} \equiv \mathcal{N}(\Jb_{\alpha\beta}, \sigma)$  for $(\alpha,\beta) \in \{1,2\}^2$.
We set $S_{\beta}(x)=\tanh(g\,x)$ and $\Ia=0,f_\alpha=0, \ta=\tau,\; \alpha=1,2$. Note that in that case, the nonlinear function $S(x)$ takes negative values. This case has a theoretical interest in the sense that analytical calculations can be pursued, which provides a good understanding of the mechanisms. Note that this case corresponds to a case where the sigmoidal transform is positive, by considering that $S(x)=\tanh(g\,x)+1$ (which is always positive), and for which the current $I_{\alpha}$ is now no more constant but equal to the random variable  $-J_{\alpha \beta}$. All the computations can be done similarly when the current is a random variable. The dynamic mean-field equation for $\mua(t)$ is given, in differential form, by:
 \begin{equation*}
 \frac{d\mua}{dt}=-\frac{\mua}{\ta}+\sum_{\beta=1}^2 \Jbab \int_{-\infty}^\infty S\left(\sqrt{v_\beta(t)}x+\mub(t)\right) Dx,\ \alpha=1,2.
 \end{equation*}
Let us denote by $G_\alpha(\mu,C(t))$ the function in the righthand side of the equality. Since $S$ is odd, $\int_{-\infty}^\infty S(\sqrt{v_\beta(t)}x) Dx =0$. Therefore, we have $G_\alpha(0,C(t))=0$ whatever $C(t)$, and hence the point $\mu_1=0, \mu_2=0$ is always a fixed point of this equation. Let us study the stability of this fixed point. To this purpose, we compute the partial derivatives of $G_{\alpha}(\mu,C(t))$ with respect to $\mub$ for $(\alpha, \beta)\in\{1,2\}^2$. We have:
\[\frac{\partial G_\alpha}{\partial \mub}(\mu,C(t))= -\frac{\delta_{\alpha\beta}}{\ta}+g \Jbab \int_{-\infty}^\infty \left(1-\tanh^2\left(\sqrt{v_\beta(t)}x+\mub(t)\right)\right) Dx,\] 
and hence at the point $\mu_1=0, \mu_2=0$, these derivatives read:
  \[\frac{\partial G_\alpha}{\partial \mub}(0,C(t))=-\frac{\delta_{\alpha\beta}}{\ta}+g \Jbab h(v_\beta(t)),\] 
where $h(v_\beta(t))=1-\int_{-\infty}^\infty \tanh^2(\sqrt{v_\beta(t)}x) Dx$.

In the case $v_\alpha(0)=0,\sigma=0,f_{\alpha}=0$, implying $v_\alpha(t)\equiv 0$ for $t \geq 0$, and the equation for
 $\mua$ reduces to:
 \begin{equation*}
 \frac{d\mua}{dt}=-\frac{\mua}{\ta}+\sum_{\beta=1}^2 \Jbab S(\mub(t))
 \end{equation*}
which is the standard Wilson and Cowan system where Gaussian fluctuations are neglected. 
In this case the stability of the fixed point $\mu=0$ is given by the sign of the largest eigenvalue of the Jacobian matrix of the system that reads:
\[\left (\begin{array}{cc} -\frac{1}{\tau_1} & 0 \\ 0 & -\frac{1}{\tau_2} \end{array}\right)+ g\left(\baR{ccc} \Jb_{11}&\Jb_{12}\\ \Jb_{21}&\Jb_{22} \eaR\right).\]
The eigenvalues  are in this case $-\frac{1}{\tau}+g\lambda_{1,2}$,
where $\lambda_{1,2}$ are the eigenvalues of $\bar{\cJ}$
and have the form:
\[\lambda_{1,2}=\frac{\Jb_{11}+\Jb_{22} \pm \sqrt{(\Jb_{11}-\Jb_{22})^2+4 \Jb_{12}\Jb_{21}}}{2}.\]
 
 Hence, they are complex whenever $\Jb_{12}\Jb_{21}< -(\Jb_{11}-\Jb_{22})^2/4$, corresponding
 to a negative feedback loop between population 1 and 2. In that case 
 they have a nonzero real part only if $\Jb_{11} + \Jb_{22} \neq 0$  (self interaction).
 This opens up the possibility to have an instability of the fixed point ($\mu=0$) leading to a
regime  where the average value of the membrane potential
 oscillates. This occurs if
   $\Jb_{11}+\Jb_{22}>0$  and if $g$ is larger than:
 \begin{equation*}
 g_c=\frac{2}{\tau(\Jb_{11}+\Jb_{22})}.
 \end{equation*}
The corresponding bifurcation is a Hopf bifurcation.

The presence of this bifurcation and these oscillations depends on the presence of fluctuations of the Gaussian field. These oscillations can simply disappear, as we study in depth in a forthcoming paper. Indeed, similarly to the case of the pitchfork bifurcation (Figure~\ref{fig:Pitch}), the covariance of the noise will globally have the effect to delay, or event destroy the Hopf bifurcation, stabilizing the fixed point $\mu=0$. 

We show here that in a case where the cycle is conserved, the correlation will interact with the mean in an intricate fashion. We observe that the correlations present oscillations four times faster than the mean. The mean asymptotically oscillates at the same frequency as Wilson and Cowan system, and in opposition of phase. 

\begin{figure}[htb]
\centerline{\includegraphics[width=.7\textwidth]{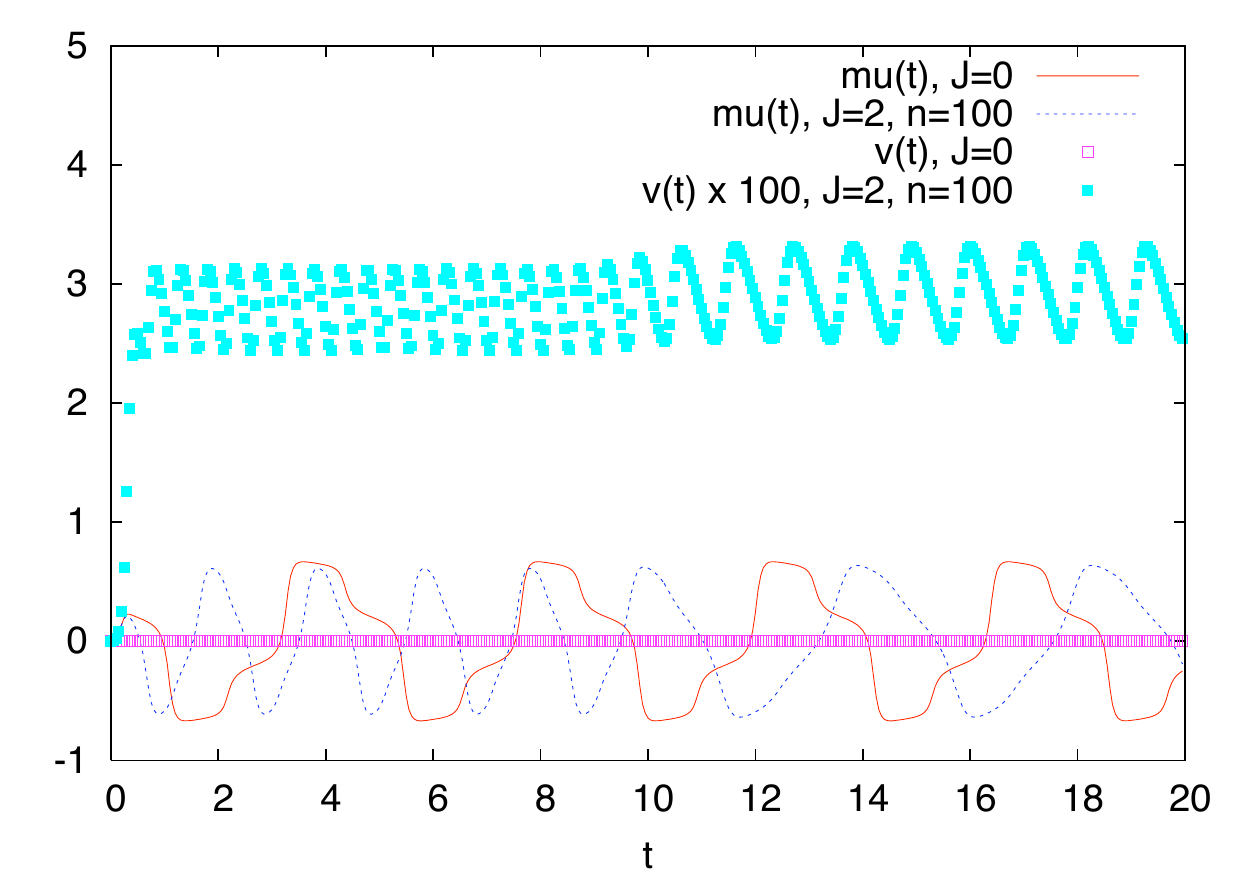}}
\caption{Evolution of the mean $\mu_1(t)$ and variance $v_1(t)$ for the mean-field of population
$1$, for $J=0$ and $J=2$, over a time window $[0,20]$. $n$ is the number of iterations the map $\mathcal{F}$. This corresponds to a number of iterations for which the method has essentially converged (up to some precision). Note that $v_1(t)$ has been magnified by a factor of $100$. Though Gaussian fluctuations are small, they have
a strong influence on $\mu_1(t)$.}
\label{F2popmu}
\end{figure}

\section{Conclusion}
We reviewed a recent approach to the mean-field limits in neural networks that takes into account the stochastic nature of input currents at each level as well as the uncertainty in synaptic coupling. One of the main advantages of this approach is that it can be rigorously derived from first principles using advanced mathematical tools in the field of probability, such as large deviation techniques and fixed points in the set of stochastic processes. However, in the simpler case treated here, the solutions are Gaussian and therefore described by the mean and covariance functions. Equations on the mean and the standard deviation are written in the simplest possible way, and a simulation algorithm is provided. 

These equations are characterized by a strong interplay between the mean and the covariance functions, that makes the solutions deviate from the customary Wilson and Cowan system. This coupling between the mean and the covariance was recently investigated independently in mean-field limits of spiking neurons using a Markovian formalism, but the equations involved appear to be simpler, and described by ordinary differential equations \cite{buice-cowan:09,bressloff:09}. The more complex integral equations that we obtain here originate from the complex memory effects that arise in mean-field limits; this has already been observed and discussed in other fields, for instance in the case of spin glasses and in the queueing theory. This phenomenon explains why the equations obtained are not Markovian. The links between the present mean-field equations and the mean-field equations studied by Bressloff, Buice, Cowan and colleagues is currently an important area of investigation. Another perspective is the precise study of mean-field equations and of how the results depart from Wilson and Cowan's system. In particular, calculations can be developed in a few particular cases, and this phenomenon provides a first approach to bifurcations in stochastic systems, which is a great endeavor in mathematics, physics and biology today. 

\appendix
\section{Convergence of the moments of the microscopic interaction process}\label{append:Moments}
In the proof of proposition \ref{prop:Convergence}, we omitted for the sake of compactness the calculations necessary to identify the covariance of the effective interaction process and its asymptotic independence with respect to the voltage process. The simple yet heavy computations are provided here. 

We start by computing the covariance between the random variables $U^N_{\alpha \beta}(t),\,U^N_{\gamma,\delta}(s)$ for $t,s>0$ and $(\alpha, \beta,\gamma,\delta) \in \{1,\ldots,P\}^4$, assuming that $N$ is big enough so that Amari's local chaos hypothesis applies. In that case, we have: 
\begin{align*}
 &\Cov\left(U^N_{\alpha \beta}(t),\,U^N_{\gamma,\delta}(s)\right) \\
 &= \mathbbm{E} \Bigg( \left\{\sum_{j=1}^{N_{\beta}} J_{i,j}S_{\beta}(V_j(t)) - \frac{\Jbab}{N_{\beta}} m_{\beta}(t)\right\}
\times \left\{\sum_{k=1}^{N_{\delta}} J_{l,k}S_{\delta}(V_k(t)) - \frac{\Jbcd}{N_{\delta}} m_{\delta}(t)\right\}\Bigg)\\
 &= \mathbbm{E} \Bigg( \left\{\sum_{j=1}^{N_{\beta}} (J_{i,j} - \frac{\Jbab}{N_{\beta}}) S_{ \beta}(V_j(t)) +  \frac{\Jbab}{N_{\beta}} (S_{ \beta}(V_j(t)) - m_{\beta}(t) )\right\} \\
 & \qquad \qquad \times \left\{\sum_{k=1}^{N_{\delta}} (J_{l,k}- \frac{\Jbcd}{N_{\delta}} ) S_{\delta}(V_k(t)) +\frac{\Jbcd}{N_{\delta}} (S_{\delta}(V_k(t)) - m_{\delta}(t)) \right\}\Bigg)
\end{align*}
This product involves four types of terms. However, because of the independence of all involved processes, this sum simply reads:
\begin{align*}
	&\mathbbm{1}_{\alpha=\gamma, \beta=\delta} \left[\Jdab \mathbbm{E}\Big(S_{ \beta}(V_{\beta}(t))S_{\beta}(V_{\beta}(s))\Big) + \frac{\Jbab^2}{N_{\beta}} \Cov(S_{\beta}(V_{\beta}(t)),S_{\beta}(V_\beta(s))) \right]\\
	&\xrightarrow[{N\to \infty}]{} \;\mathbbm{1}_{\alpha=\gamma, \beta=\delta}\;\Jdab\; \Delta_{\beta}^V(t,s)
\end{align*}

We therefore proved the existence of a limit to the covariance of the processes, which proves the convergence of the microscopic interaction process to the effective interaction process, by application of the functional central limit theorem. 

In order to prove that the effective interaction process is independent of the membrane potential $V$, we use the fact that both processes are Gaussian, and therefore that independence is equivalent to a null correlation between the two processes. Let us consider a neuron $i$ of class $\gamma$, $\alpha, \beta \in \{1,\ldots, P\}$, and $k$ a neuron of population $\alpha$. We have:

\begin{align*}
 & \Cov(V_i(t), U^N_{\alpha, \beta}(s)) = \Cov( V_i(t), \sum_{j=1}^{N_{\beta}} J_{k,j} S_{\beta}(V_j(s))) \\
 & \quad = \sum_{j=1}^{N_{\beta}} \mathbbm{E}\left( \left(V_i(t)-\mu_i(t)\right) \left(J_{kj} S_{\beta}(V_j(s)) - \frac{\Jbab}{N_{\beta}} m_{\beta}^V(s) \right)\right) \\
 &\quad  = \mathbbm{1}_{\alpha=\gamma} \frac{\Jbab}{N_{\beta}} \Cov(V_\alpha,S(V_\alpha)).
\end{align*}
which tends to zero when the number of neurons tends to infinity since the covariance can be easily proved to be bounded. 

\section{Case of $\erf$ firing-rate transformations}\label{append:Erf}
In this appendix we provide the computations leading to simplify the mean-field equations when the sigmoidal functions are erf functions $S_{i}(x) = \erf(g_{p(i)}x+\gamma_{p(i)})$. In that case, the mean equation \eqref{eq:moyenne}:

\begin{align*}
  \mathbbm{E}(g_{\alpha}(X_{\alpha}))(t)&=\int_{\R} erf\left(g_{p(i)}\left(x \sqrt{v_{\beta}(t)}+\mu_\beta(t)\right)+\gamma_{p(i)}\right) \frac{e^{-x^{2}/2}}{\sqrt{2\pi}}\,dx\\
  & = \int_{\R}\int_{-\infty}^{g_{p(i)}\left(x \sqrt{v_{\beta}(t)}+\mu_\beta(t)\right)+\gamma_{p(i)}}\frac{e^{-(x^{2}+y^{2})}}{2\pi} dx dx'
\end{align*}

Let us write this equation with generic boundaries
\[\int_{\R}\int_{-\infty}^{a\,x +b}\frac{e^{-(x^{2}+y^{2})/2}}{2\pi} dx dy\]
The integration domain has therefore an affine shape as plotted in figure \ref{fig:ChangeVar}. We make the change of variables consisting of a rotation of the axis and making the boundary of the domain parallel to one of the new axis (rotating from $(x,y)$ to $(u,v)$, see figure \ref{fig:ChangeVar}). The rotation is therefore of angle $\alpha$ such that $tan(\alpha)=g\,a$. The new integration domain is in the new coordinates given by $v\leq \frac{b \sin(\alpha)}{a} = \frac{g\,b}{\sqrt{1+g^{2}a^{2}}}$:

\begin{figure}
\begin{center}
	\includegraphics[width=.3\textwidth]{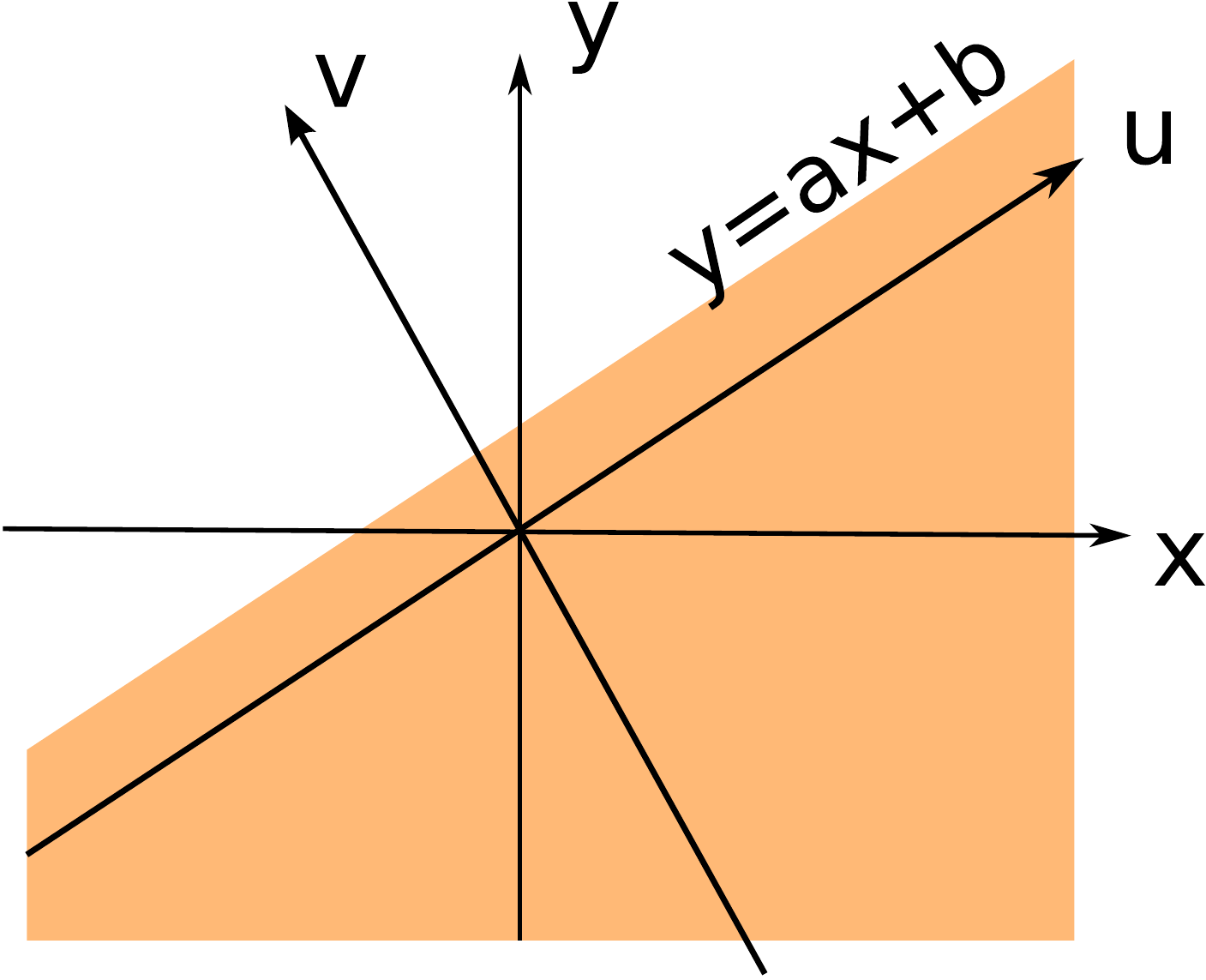}
\end{center}
\caption{Change of variable for the erf function.}
\label{fig:ChangeVar}
\end{figure} 

\begin{align*}
  \int_{\R}\int_{-\infty}^{a\,x +b}\frac{e^{-(x^{2}+y^{2})/2}}{2\pi} dx dy &=\int_{\R}\int_{-\infty}^{\frac{g\,b}{\sqrt{1+g^{2}a^{2}}}} e^{-(u^{2}+v^{2})/2} \frac{1}{2\pi} du dv\\
  & = \erf\left(\frac{gb}{\sqrt{1+g^{2}a^{2}}}\right)
\end{align*}
which reads with the parameters of the model:
\begin{equation}\label{eq:M-Erf}
	m_\beta^X(t) = \erf\left(\frac{g_{p(i)}\, \mu_\beta(t) + \gamma_{p(i)}}{\sqrt{1+g^{2}v_{\beta}(t)}}\right)
\end{equation}

The correlation function is simplified along the same lines. Indeed, the term $ \Debx(t,s)$ can be written as:

\begin{align*}
    \Debx(t,s) &=  \int_{\R^{2}} S(a\,x + b\,y+c) \, S(d\,y + e) Dx Dy\\
    & = \int_{R} S(d\,y+e) \int_{\R} S(a\,x+b\,y+c) Dx Dy\\
    & = \int_{\R} S(d\,y+e) S\left( \frac{by+c}{\sqrt{1+g^{2}a^{2}}}\right) Dy
\end{align*}
which, in terms of the model's parameters, read:
\begin{equation}\label{eq:C-Erf}
  \Debx(t,s)=\int_{\R} S(\sqrt{v_{\beta}(t)} y + \mu_{\beta}(t)) \,S\left ( \frac{C_{\beta\beta}(t,s) y + \mu_{\beta}(s)\sqrt{v_{\beta}(t)} } {v_{\beta}(t) + g^{2} (v_{\beta}(t)v_{\beta}(s)-C_{\beta \beta}(t,s)^2)}\right)Dy 
\end{equation}

\bibliographystyle{plain}


\end{document}

%% file: macros.tex
\theoremstyle{plain}
\newtheorem{theorem}{Theorem}[section]

\theoremstyle{definition}

\newtheorem{proposition}[theorem]{Proposition}

\newtheorem*{remark}{Remark}

\theoremstyle{remark}

\newcommand{\bd}{\begin{document}}
\newcommand{\ed}{\end{document}}
\newcommand{\beq}{\begin{equation}}
\newcommand{\eeq}{\end{equation}}
\newcommand{\nid}{\noindent}
\newcommand{\ben}{\begin{enumerate}}
\newcommand{\een}{\end{enumerate}}
\newcommand{\bit}{\begin{itemize}}
\newcommand{\eit}{\end{itemize}}
\newcommand{\baR}{\begin{array}}
\newcommand{\eaR}{\end{array}}
\newcommand{\su}{\section}
\newcommand{\ssu}{\subsection}
\newcommand{\sssu}{\subsubsection}
\newcommand{\deq}{\stackrel {\rm def}{=}}
\newcommand{\bea}{\begin{eqnarray}}
\newcommand{\eea}{\end{eqnarray}}
\newcommand{\mZ}{\mathbf{Z}}
\newcommand{\mg}{\mathbf{g}}
\newcommand{\mm}{\mathbf{m}}

\newcommand {\VV} {\mathbf{V}}
\newcommand {\V} {\mathcal{V}}
\newcommand {\G} {\mathcal{G}}
\newcommand{\R}{\mathbbm{R}}
\newcommand{\N}{\mathbbm{N}}
\newcommand{\Lop}{\mathcal{L}}
\newcommand{\F}{\mathcal{F}}
\newcommand{\Pro}{\mathbbm{P}}
\newcommand{\f}{\varphi}
\newcommand{\Diag}{{\mathbf{\rm Diag}}}
\newcommand{\ty}{\tilde{y}}
\newcommand{\Jb}{\bar{J}}
\newcommand{\Jbab}{\bar{J}_{\alpha\beta}}
\newcommand{\Jbcd}{\bar{J}_{\gamma\delta}}
\newcommand{\Jab}{\sigma_{\alpha\beta}}
\newcommand{\Jdab}{\sigma^2_{\alpha\beta}}
\newcommand{\ma}{m_\alpha}
\newcommand{\Ma}{M_\alpha}
\newcommand{\mua}{\mu_\alpha}
\newcommand{\muax}{\mua^X}
\newcommand{\muay}{\mua^Y}
\newcommand{\va}{v_\alpha}
\newcommand{\vax}{v_\alpha^X}
\newcommand{\Ca}{C_{\alpha\alpha}}
\newcommand{\Cax}{C_{\alpha\alpha}^X}
\newcommand{\Cbx}{C_{\beta\beta}^X}
\newcommand{\Cay}{C_{\alpha\alpha}^Y}
\newcommand{\Ia}{I_\alpha}
\newcommand{\Da}{\Delta_\alpha}
\newcommand{\Deab}{\Delta_{\alpha\beta}}
\newcommand{\Deabx}{\Delta_{\alpha\beta}^X}
\newcommand{\Debx}{\Delta_{\beta}^X}
\newcommand{\sa}{s_\alpha}
\newcommand{\ta}{\tau_\alpha}
\newcommand{\mub}{\mu_\beta}
\newcommand{\mb}{m_{\alpha\beta}}
\newcommand{\vb}{v_\beta}
\newcommand{\vbx}{v_{\beta}^X}
\newcommand{\Cb}{C_{\beta\beta}}
\newcommand{\Db}{\Delta_\beta}
\newcommand{\Mb}{M_\beta}
\newcommand{\qb}{q_\beta}
\newcommand{\gb}{\gamma_\beta(\mub,\vb;\Cb)}
\newcommand{\gbt}{\gamma_\beta(\mub,\vb;\Cb(t))}
\newcommand{\tb}{\tau_\beta}
\newcommand{\Sb}{S_\beta}
\newcommand{\Ha}{H_{\alpha}}
\newcommand{\Hax}{H_{\alpha}^X}
\newcommand{\Hab}{H_{\alpha\beta}}
\newcommand{\Habx}{H_{\alpha\beta}^X}
\newcommand{\Dab}{D_{\alpha\beta}}
\newcommand{\Dabx}{D_{\alpha\beta}^X}
\newcommand{\cJ}{{\cal J}}

\newcommand\bbbr{{\sf I\!R}}

\newcommand{\ind}[1]{\mathbbm{1}_{#1}}          

\newcommand{\derpart}[2]{ \frac{\partial #1}{\partial #2} } 
\newcommand{\der}[2]{ \frac{\text{d} #1}{\text{d} #2} }  
\newcommand{\derparts}[2]{ \frac{\partial^2 #1}{\partial #2 ^2}}
\newcommand{\derpartsxy}[3]{ \frac{\partial^2 #1}{\partial #2 \partial #3}} 
\newcommand{\argmin}{\mathop{\mathrm{Arg\,Min}}}
\newcommand{\eps}{\varepsilon}

\DeclareMathOperator{\eqlaw}{\stackrel{\mathcal{L}}{=}}         

\newcommand{\Herm}[1]{\mathcal{H}_{#1}}            
\newcommand{\Hermnu}{\mathcal{H}_{\nu}}            
\newcommand{\erf}{\mathrm{erf}}              
\newcommand{\Cov}{\mathrm{Cov}}

\newcommand{\m}{\mathcal}
\newcommand{\Ie}{I_{\text{ext}}}

\newcommand {\image} {\includegraphics}

\newcommand{\mC}{\mathbf{C}}
\newcommand{\mD}{\mathbf{D}}
\newcommand{\mE}{\mathbf{E}}
\newcommand{\mX}{\mathbf{X}}
\newcommand{\mx}{\mathbf{x}}
\newcommand{\mY}{\mathbf{Y}}
\newcommand{\mL}{\mathbf{L}}
\newcommand{\mQ}{\mathbf{Q}}
\newcommand{\mI}{\mathbf{I}}
\newcommand{\mJ}{\mathbf{J}}
\newcommand{\mW}{\mathbf{W}}
\newcommand{\mU}{\mathbf{U}}
\newcommand{\mone}{\mathbf{1}}

\newcommand{\msig}{\boldsymbol{\sigma}}
\newcommand{\mtheta}{\boldsymbol{\theta}}
\newcommand{\mdelta}{\boldsymbol{\Delta}}
\newcommand{\mla}{\boldsymbol{F}}
\newcommand{\mga}{\boldsymbol{\Gamma}}
\newcommand{\mxi}{\boldsymbol{\xi}}

\newcommand{\X}{\mathcal{X}}
\newcommand{\proc}{\mathcal{M}_1^+(C([t_0,T], \R)}
\newcommand{\procP}[1]{\mathcal{M}_1^+(C([t_0,T], \R^{#1}))}
\newcommand{\proci}{\mathcal{M}_1^+(C((-\infty,T], \R)}
\newcommand{\prociP}[1]{\mathcal{M}_1^+(C((-\infty,T], \R^{#1})}
\newcommand{\procs}[1]{\mathcal{M}_1^+(C([0,#1], \R)}
\newcommand{\Exp}[1]{\mathbb{E}\left [ #1 \right]}
\newcommand{\ExpSimple}{\mathbb{E}}
\newcommand{\nvar}[1]{\left \| #1 \right\|_{\mathrm{var}}}
\newcommand{\nvarP}[2]{\left \| #1 \right\|_{\mathrm{var}, #2}}
\newcommand{\abs}[1]{\left \vert #1 \right \vert }
\newcommand{\pente}{\|S_{\beta}'\|_{\infty}}
\newcommand{\Fs}{\mathcal{F}_{\text{stat}}}
\newcommand{\triplebar}[1]{\vert\vert\vert #1 \vert\vert\vert}
\newcommand{\procsP}[2]{\mathcal{M}_1^+(C([t_0,#1], \R^{#2})}
\newcommand{\dvar}[2]{d_{\rm{var}}\left(#1,\; #2\right)}
\newcommand{\dt}[2]{d_t\left(#1,\; #2\right)}
\newcommand{\norm}[1]{\left \| #1 \right \|}
\newcommand{\Si}{S}
\newcommand{\T}{\mathcal{T}}
\newcommand{\IT}{\mathcal{IT}}
\newcommand{\tV}{\mathbf{\widetilde{V}}}

\newcommand{\mA}{\mathbf{B}}

\newenvironment{heuriproof}{\noindent \emph{Heuristic proof.}}{\begin{flushright} \qed \end{flushright}}
\newcommand{\eqdef}{\overset{\rm def}{=}}